\documentclass[12pt,letterpaper, reqno]{amsart}

\usepackage{amssymb,latexsym, amsmath, amsxtra}
\usepackage[dvips]{graphics}

\theoremstyle{plain}
\newtheorem{theorem}{Theorem}[section]

\newtheorem{question}[theorem]{Question}
\newtheorem{lemma}[theorem]{Lemma}

\theoremstyle{definition}
\newtheorem{definition}[theorem]{Definition}
\theoremstyle{remark}

\newtheorem*{remark}{Remark}
\numberwithin{equation}{section}
\numberwithin{theorem}{section}
\numberwithin{table}{section}
\numberwithin{figure}{section}

\newcommand{\R}{\mathbb R}

\newcommand{\Q}{\mathbb Q}

\def\({\left(}
\def\){\right)}

\begin{document}
\title{The highest lowest zero of general L-functions}
\author[Bober, et al]{Jonathan Bober, J.~Brian~Conrey, David~W.~Farmer, Akio~Fujii, Sally~Koutsoliotas, Stefan~Lemurell, Michael~Rubinstein, and  Hiroyuki~Yoshida}

%\address{
%{\parskip 0pt
%American Institute of Mathematics\endgraf
%farmer@aimath.org\endgraf
%}
%  }

\begin{abstract}
Stephen D. Miller showed that, assuming the generalized Riemann Hypothesis,
every entire $L$-function of real archimedian type has a zero in the
interval $\frac12+i t$ with $-t_0 < t < t_0$, where $t_0\approx 14.13$ 
corresponds to the first zero of the Riemann zeta function.
We give a numerical example of a self-dual degree-4 $L$-function
whose first positive imaginary zero is at $t_1\approx 14.496$.  In particular,
Miller's result does not hold for general $L$-functions.  
We show that all $L$-functions 
satisfying some additional (conjecturally true) conditions have
a zero in the interval $(-t_2,t_2)$ with $t_2\approx 22.661$.
\end{abstract}

\maketitle

\section{Introduction} 

The Generalized Riemann Hypothesis asserts that the nontrivial zeros
of an $L$-function lie on the critical line $\Re(s)=\frac12$.
Thus, the zeros can be listed as $\frac12+i \gamma_n$ where
$\cdots\le \gamma_{-2}\le \gamma_{-1}<0\le \gamma_1\le \gamma_2\le \cdots$,
with zeros repeated according to their multiplicity.  We refer to
$\frac12+i\gamma_1 $, or to $\gamma_1$ when no confusion will result,
as the ``first'' zero of the $L$-function.

The first zero of the Riemann zeta function is approximately
$\gamma_1\approx 14.13$.  Stephen D. Miller\cite{M} proved that a large
class of $L$-functions have a smaller first zero, so among that
class the zeta function has the highest lowest zero.  Miller was motivated
by a question of Odlyzko\cite{O}, who showed that the Dedekind
zeta function of any number field has a zero whose imaginary part is
less than~14.

In this note we quote results from \cite{FKL} and \cite{B}, respectively, 
which establish the following:

\begin{itemize}
\item{} There exists an $L$-function whose lowest zero is higher than the
lowest zero of the Riemann zeta function.

\item{} Assuming certain generally believed hypotheses, there is a universal
upper bound on the gap between consecutive critical zeros of any $L$-function.
\end{itemize}

In the next section we briefly introduce the $L$-functions 
we consider in this paper and describe a general
result bounding the gaps between consecutive zeros.  Then in Section~\ref{sec:example}
we give a numerical example from \cite{FKL} of  an $L$-function with $\gamma_1\approx 14.496$ and we
explain why it is not actually that surprising that there are $L$-functions
whose first zero is higher than that of the Riemann zeta function.
Then in Section~\ref{sec:explicit} we use the explicit formula
to give upper bounds for $\gamma_1$ for the $L$-functions we consider.
In Section~\ref{sec:millermethod} we show that, while it is not
surprising that there are $L$-functions which have large gaps between their
zeros, the existence of the example in Section~\ref{sec:example} is
surprising.

\section{$L$-functions}\label{sec:Lfunctions}

By an \emph{$L$-function} we mean the $L$-function attached to
an irreducible unitary cuspidal automorphic representation of $GL_n$ over $\Q$,
and furthermore we assume the Ramanujan-Petersson conjecture and the
Generalized Riemann Hypothesis.  This means that we can write the
$L$-function as
a Dirichlet series
\begin{equation}\label{eqn:ds}
L(s) = \sum_{n=1}^\infty \frac{a_n}{n^s}
\end{equation}
where $a_n\ll n^\delta$ for any $\delta>0$, which has an Euler
product
\begin{equation}\label{eqn:ep}
L(s)=\prod_p L_p(p^{-s})^{-1}
\end{equation}
and satisfies a functional equation of the form
\begin{equation}\label{eqn:fe}
    \Lambda(s) =  Q^s \prod_{j=1}^d \Gamma_\R\left(s + \mu_j\right) L(s) 
= \varepsilon \overline{\Lambda(1 - \bar{s})}.
\end{equation}
Here $|\varepsilon|=1$ and we assume that $\Re(\mu_j)\ge 0$ and~$Q\ge 1$.
The normalized $\Gamma$-function is defined as
\begin{equation}
\Gamma_\R(s) = \pi^{-s/2}\Gamma(s/2),
\end{equation}
where $\Gamma(s)$ is the usual Euler Gamma function.
The number $d$ is called the \emph{degree} of the $L$-function, which
for all but finitely many $p$ is also the degree of the polynomial~$L_p$.

We use Weil's explicit formula, given in Lemma~\ref{lem:weil},
 to prove the following theorem.

\begin{theorem}\label{thm:selberg-class-bound}
If $L(s)$ is entire and satisfies the Generalized Riemann Hypothesis,
then $L(1/2 + it)$ has a zero
in every interval of the form $t \in [t_0, t_0 + 45.3236]$.
\end{theorem}

In the case all $\mu_j$ are real, Miller~\cite{M} proved the above theorem
with ``45.3236'' replaced by ``28''.   In Section~\ref{sec:example} we give a numerical example
to illustrate why things behave differently when the $\mu_j$ are complex.
That example has $\gamma_1-\gamma_{-1}\approx 28.992$.
A slightly improved version of Theorem~\ref{thm:selberg-class-bound}
is given by Bober~\cite{B}, and he also gives the optimal result
for the cases $d=3$ and~$4$.

The term ``lowest zero'' of an $L$-function is ill-defined, because
one must first choose a normalization of the $L$-function.  
The normalization is
clear in the case of Miller~\cite{M} because the parameters in the
$\Gamma$-factors can be chosen to be real.  But if $L(s)$ is an
$L$-function then so is $L(s+i y)$ for any real $y$.  A reasonable
normalization is to require $\sum \Im(\mu_j) = 0$, but other
normalizations are possible.  Thus, it is
natural to consider the maximum possible gap between
zeros instead, which is how we phrased our result above.

This discussion suggests two questions:

\begin{question}\label{q:1} Does there exist an $L$-function with a larger gap between its
zeros than any other $L$-function?
\end{question}

Theorem~\ref{thm:selberg-class-bound} shows that there is a least upper bound,
$\Upsilon$,
on the gap between consecutive zeros; the question is whether that bound
is attained.  We do not have a conjecture for $\Upsilon$, but Bober~\cite{B}
suggests that $\Upsilon < 36$ (i.e., $\gamma_1 < 18$).

\begin{question}\label{q:2} If $0<u<\Upsilon$, does there exist an $L$-function whose largest
zero gap is arbitrarily close to~$u$?
\end{question}

Considerations of the function field analogue and a conjecture of Yoshida~\cite{Y} suggest that the answers to Questions~\ref{q:1} and~\ref{q:2} may be
'no' and 'yes,' respectively.

Most of this work on this paper was completed during the workshop
\emph{Higher rank $L$-functions: theory and computation},
held at the Centro de Ciencias de Benasque Pedro Pascual in July 2009.
The motivation was a suggestion by David Farmer that one could make an ordered
list of all $L$-functions according to their lowest critical zero.
It was disappointing to find that such an ordering does not place
the Riemann zeta function first, and in fact the $L$-function of the
Ramanujan $\tau$-function would come before all the Dirichlet $L$-functions.
Depending on the answers to Questions~\ref{q:1} and~\ref{q:2},
it is possible that this ``list''
would not have a first
element, and any two $L$-functions could actually have infinitely many other
$L$-functions between them.

\section{A certain degree-4 $L$-function}\label{sec:example}

In~\cite{FKL} the authors perform computational experiments to discover
$L$-functions with functional equation~\eqref{eqn:fe} with $d=3$ or~$4$
and the $\mu_j$ purely imaginary.  The results are approximate values
for the $\mu_j$ and the coefficients $a_n$, which are claimed to be accurate
to several decimal places.  While it is not currently possible to
prove that those numerical examples are indeed approximations to actual $L$-functions, the
functions pass several tests which lend credence to their claim.

One example which is relevant to the present paper has $d=4$ with
$\mu_1=-\mu_2=4.7209 i$ and $\mu_3=-\mu_4=12.4687 i$.  Appropriately
interpreted, this is the ``first'' $L$-function with $d=4$ and the
$\mu_j$ purely imaginary.  A plot of the $Z$-function along the critical
line is given in Figure~\ref{fig:degree4}.  Note that on the
critical line the $Z$-function has the same absolute value as the $L$-function;
in particular, it has the same critical zeros.

 Figure~\ref{fig:degree4} shows that this $L$-function has its first
zero at $\gamma_1=14.496$.  The $L$-function is self-dual, so
$Z(t)$ is an even function of $t$ and
$\gamma_{-1}=-14.496$, giving a gap between zeros of 28.992.

\begin{figure}[htp]
%\begin{center}
\scalebox{1.0}[1.0]{\includegraphics{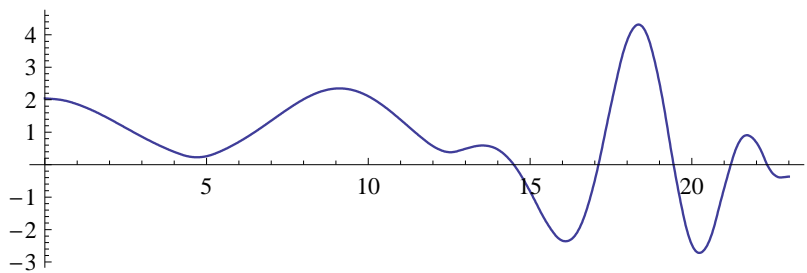}}
\caption{\sf
The $Z$-function of an $L$-function satisfying functional equation~\eqref{eqn:fe}
with $d=4$ and $\mu_1=-\mu_2=4.7209 i$ and $\mu_3=-\mu_4=12.4687 i$.
The first zeros are at $\pm \gamma_1=14.496$.
Data taken from~\cite{FKL}.
} \label{fig:degree4}
%\end{center}
\end{figure}

The plot in Figure~\ref{fig:degree4} shows local minima near 4.7
and 12.5.  Those are due to the trivial zeros, which have
imaginary parts $- \Im(\mu_j)$.  Those trivial zeros suppress
the appearance of nearby zeros on the critical line, a phenomenon
first observed by Strombergsson~\cite{St}.
Thus, such $L$-functions can have a surprisingly large gap between
their critical zeros.

For the reader who may wish to check our calculations, such as with the
explicit formula, we provide the spectral parameters and initial
Dirichlet coefficients to higher precision:
\begin{align}
\mu_1=\mathstrut &4.72095103638565339773\cr
\mu_2=\mathstrut &12.4687522615131728082\cr
&\mathstrut \cr
a_2=\mathstrut&   \phantom{\mathstrut -\mathstrut }1.34260324197021624329 \cr
a_3=\mathstrut&  -0.18745190876087089719 \cr
a_4=\mathstrut&   \phantom{\mathstrut-\mathstrut}0.4644565335271682550 \cr
a_5=\mathstrut&  -0.001627934631772515 \cr
a_7=\mathstrut&   \phantom{\mathstrut-\mathstrut}0.22822958260580737 \cr
a_9=\mathstrut&  -0.4634288260750947 \cr
a_{11}=\mathstrut&  \phantom{\mathstrut-\mathstrut}0.695834471444353 \cr
a_{13}=\mathstrut& -0.8824356594477 
\end{align}
Note also that the zeros with imaginary part $0<\gamma<30$ are at
the heights $\{$14.4960615091, 17.1144514545, 19.4393573576,
21.193378013, 22.396088469, 23.108950059,
24.34252975, 25.59506020, 27.12281351,
28.2791393, 29.5857431$\}$.

\section{An upper bound on gaps between zeros}\label{sec:explicit}

\subsection{The explicit formula} We use Weil's explicit formula with a particular test function to
establish Theorem~\ref{thm:selberg-class-bound}.

The form of the explicit formula that we will use is the following.

\begin{lemma}\label{lem:weil} Suppose that $L(s)$ has a Dirichlet series
expansion
\eqref{eqn:ds} which continues to an entire function such that
\begin{equation}
    \Lambda(s) =
		Q^s \prod_{j=1}^d \Gamma_\R\left(s + \mu_j\right) L(s)
	= \varepsilon \overline{\Lambda(1-\overline{s})}
\end{equation}
is entire and satisfies the mild growth condition
$L(\sigma + it) \ll |t^A|$, uniformly in $t$ for bounded~$\sigma$.
Let $f(s)$ be holomorphic in a horizontal strip $-(1/2 + \delta) < \Im(s) < 1/2 + \delta$ with $f(s) \ll \min(1, |s|^{-(1+\epsilon)})$
in this region, and suppose that $f(x)$ is real valued for real $x$.
Suppose also that the Fourier transform
of $f$ defined by
\[
    \hat f(x) = \int_{-\infty}^\infty f(u)e^{-2\pi i u x} dx
\]
is such that
\[
    \sum_{n=1}^{\infty} \frac{c(n)}{n^{1/2}} \hat{f} \left( \frac{\log{n}}{2 \pi} \right) + \frac{\overline{c(n)}}{n^{1/2}} \hat f\left( -\frac{\log n}{2 \pi}\right)
\]
converges absolutely, where $c(n)$ is defined by
\begin{equation}
\frac{L'}{L}(s) = \sum_{n=1}^{\infty} \frac{c(n)}{n^s} .
\end{equation}
Then
\begin{align} \label{weil}
\sum_{\gamma} f(\gamma) =\mathstrut &  \frac{\widehat{f}(0)}{\pi} \log{Q} + \frac{1}{2 \pi}
\sum_{j=1}^d \ell(\mu_i, f)\cr
&+ \frac{1}{2 \pi} \sum_{n=1}^{\infty} \frac{c(n)}{n^{1/2}} \hat{f} \left( \frac{\log{n}}{2 \pi} \right) + \frac{\overline{c(n)}}{n^{1/2}} \hat f\left( -\frac{\log n}{2 \pi}\right)
\end{align}
where
\begin{equation}
\ell(\mu, f) = \Re\left\{\int_\R \frac{\Gamma'}{\Gamma} \left( \frac{1}{2} \left( \frac{1}{2} + i t \right) + \mu \right) f(t) dt\right\} - \hat f(0)\log \pi
\end{equation}
and the sum $\sum_\gamma$ runs over all non-trivial zeroes of $L(s)$.
\end{lemma}

\begin{proof}
This can be found in Iwaniec and Kowalski \cite[Page 109]{IK}, but note that they
use a different
normalization for the Fourier transform.
\end{proof}

Note that if we assume the Ramanujan-Petersson conjecture then $c(n) \ll n^\epsilon$,
but any mild growth estimate on the $c(n)$ is sufficient for our purposes.

The general strategy we will use is as follows: to show that $L(1/2 + it)$ has a zero
with $\alpha \le t \le \beta$,
we want to take $f$ to be a good approximation of $\chi_{(\alpha, \beta)}$, the step function with value $1$ on $(\alpha, \beta)$ and $0$ elsewhere, and such that
the support of $\hat{f}$ is contained in the interval $\left(- \frac{\log{2}}{2 \pi}, \frac{\log{2}}{2 \pi} \right)$.
Then, the last sum on the RHS of the explicit formula disappears, and for the $L$-functions that we are considering, (\ref{weil}) should look like
\begin{equation}\label{eq:misc1}
\sum_{\alpha < \gamma < \beta} f(\gamma) \approx \frac{\log Q}{\pi} \hat{f}(0) + \frac{1}{2 \pi} \sum_{j=1}^d \ell(\mu_j, f).
\end{equation}
Since $f$ approximates $\chi_{(\alpha, \beta)}$, we expect that
\[
    \ell(\mu_j, f) \approx \Re \left\{\int_{\alpha}^{\beta} \left(\frac{\Gamma'}{\Gamma} \left( \frac{1}{4} + \frac{it}{2} + \mu_j \right) \right) dt\right\} - (\beta - \alpha)\log \pi.
\]
If $\beta-\alpha$ is large enough then this will be positive for any $\mu_j$. We will then find that
the right side of \eqref{eq:misc1} is positive, which shows the existence of the zero
that we are looking for.

While we cannot actually use the characteristic function of the interval $(\alpha,\beta)$ in the explicit
formula, we do not quite need to. As long as $f(x)$ is positive for $\alpha < x < \beta$ and negative
elsewhere, the same argument will work. The function which we use here is
the Selberg minorant $S_-(z)$ for the interval $(\alpha,\beta)$ and with Fourier transform supported in
$( -(\log 2)/2\pi, (\log 2)/2\pi )$. We describe this function below. Note that with this approach
there are fundamental limits to how small we can make $\beta - \alpha$. According
to the uncertainty principle, we should need to make the support of $\hat f$ large
if we want to get a good function with $\beta - \alpha$ small.

\newcommand{\sgn}{\mathrm{sgn}}

\subsection{Selberg's amazing functions}
As we have already described, we would like to use in the explicit formula a function
$f(x)$ which is positive only in a prescribed interval and which has a compactly supported
Fourier transform. Additionally, we have some reason to believe that a good candidate
for our purposes should be close to $1$ inside this interval and close to $0$ outside
of it. Selberg \cite[pages 213--225]{Se} gives a construction of such functions which are
suitable for our purposes. 

These functions are easiest to describe by first defining the Beurling function
\[
    B(z) = 1 + 2\left(\frac{\sin \pi z}{\pi}\right)^2\left(\frac{1}{z} - \sum_{n=1}^\infty \frac{1}{(n + z)^2}\right).
\]
This function is a good approximation for the function
\[
    \sgn(x) = \left\{ \begin{array}{cl}
            1 & \textrm{if } x > 0 \\
            0 & \textrm{if } x = 0 \\
            -1 & \textrm{if } x < 0
            \end{array}\right.
\]
and it is a majorant for $\sgn(x)$; that is, $\sgn(x) \le B(x)$ for all real $x$. Beurling (unpublished)
showed that this is the best possible such approximation in the sense that if $F(z)$
is any entire function satisfying $\sgn(x) \le F(x)$ for all real $x$ and $F(z) \ll_\epsilon \exp( (2\pi + \epsilon) |z|)$, then
\[
    \int_{-\infty}^\infty \left(F(x) - \sgn(x)\right) \  dx \ge 1,
\]
with equality achieved if and only if $F(x) = B(x)$. (A proof can be found in \cite{V}.)

To approximate the characteristic function of an interval, we can use a simple linear combination
of Beurling functions.
\begin{definition}
The Selberg minorant $S_-(z)$ for the interval $[\alpha, \beta]$ and parameter $\delta > 0$
is defined by
\[\label{eqn:selminus}
    S_-(z) = -\frac{1}{2} \Big(B(\delta(\alpha - z)) + B(\delta(z - \beta))\Big)
\]
\end{definition}
Selberg \cite[pages 213--225]{Se} proved that whenever $\delta(\beta - \alpha)$ is an integer, $S_-(z)$ is a best possible minorant
for the characteristic function of the interval $[\alpha,\beta]$, in the same
sense that $B(z)$ is the best possible majorant for the $\sgn$ function, although $S_-(z)$ is not the unique best possible
minorant. We do not make us of this extremal property anywhere, but it is the motivation
behind our choice of using $S_-(z)$ in the explicit formula, and it may give
hope that our results are not too far from optimal.

We summarize some important properties of $S_-(z)$ that we do need in the following lemma.
\begin{lemma}\label{selberg-minorant-lemma} Let $S_-(z)$ be the Selberg minorant for the interval $[\alpha, \beta]$ with
parameter $\delta$. Then the following hold.
\begin{enumerate}
\item $S_-(x) \le \chi_{(\alpha, \beta)}(x)$ for all real $x$.
\item $\int_{-\infty}^\infty S_-(x) dx = \beta - \alpha - \frac{1}{\delta}$.
\item $\hat S_-(x) = 0$ for $x > \delta$ or $x < -\delta$.
\item For any $\epsilon > 0$, $S_-(z) \ll_{\delta, \alpha, \beta, \epsilon} \min\left(1, \frac{1}{|z|^2}\right)$ for $\Im(z) \le \epsilon$.
\end{enumerate}
\end{lemma}
\begin{proof}
All of these facts can be found in Selberg's work \cite[pages 213--225]{Se}.
\end{proof}

\begin{remark}
For the function $f$ that we choose in the explicit formula, we will also need $f(x) > 0$ in a prescribed
range.
If $\delta(\beta - \alpha)$ is too small, then this might not be the case for the function $S_-(x)$.
With the specific parameters we choose, this will hold for our application, however.
\end{remark}

\subsection{Proof of Theorem \ref{thm:selberg-class-bound}}

\begin{proof}[Sketch of proof of Theorem \ref{thm:selberg-class-bound}]
Lemma \ref{selberg-minorant-lemma} tells us that in the formula we may choose $f(s) = S_-(s)$.
We do so, with $\alpha = -2.5/\delta$ and $\beta = 2.5/\delta$, where $\delta = \frac{\log 2}{2\pi}$.
The explicit formula then reads
\begin{equation}\label{eq-selberg-function-explicit-formula}
    \sum_\gamma S_-(\gamma) = \frac{\log Q}{\pi} \hat S_-(0) 
        + \frac{1}{2\pi} \sum_{j=1}^d \Re\left\{ \int_{-\infty}^\infty \frac{\Gamma'}{\Gamma}\left(\frac{1}{4} + \frac{it}{2} + \mu_j\right)S_-(t) dt\right\}
        - \frac{d}{2\pi}\hat S_-(0)\log \pi
\end{equation}
Since 
$\hat S_-(0) = 4/\delta$ is positive and $Q\ge 1$,
we may ignore the first term of this sum in
establishing a lower bound. We then can check that
\[
    \Re\left\{\int_{-\infty}^\infty \frac{\Gamma'}{\Gamma}\left(\frac{1}{4} + \frac{it}{2} + \mu\right)S_-(t) dt\right\} > \hat S_-(0) \log \pi
\]
for all choices of $\mu$. The right hand side of \eqref{eq-selberg-function-explicit-formula}
is thus positive. As $S_-(\gamma)$ is only positive when $\alpha < \gamma < \beta$, we conclude
that $L(1/2 + i\gamma) = 0$ for some $\gamma$ in this range.

More details of this computation will appear in \cite{B}.
\end{proof}

\begin{remark}
Note that $\beta - \alpha \approx 45.3236$. It should be possible to use the Selberg functions to make
this difference a very little bit smaller without changing the proof, but not much. We have
chosen $\alpha$ and $\beta$ as above because the Selberg function has a much nicer representation when $(\beta - \alpha)\delta$
is an integer, which simplifies computation.
\end{remark}

\section{Positivity and non-existence of $L$-functions}\label{sec:millermethod}

Although Miller's paper~\cite{M} concerns $L$-functions of real archimedean
type, the methods also apply to the $L$-functions considered
here.  When that approach is used on degree 3 or degree 4 $L$-functions
with functional equation~\eqref{eqn:fe} and $\mu_j$ pure imaginary,
the result is not quite conclusive.  Miller's implementation involves two calculations.
One calculation shows that if all the $\mu_j$ are sufficient small
(i.e., lying in a certain bounded region which can be made explicit) then
such an $L$-function cannot exist.  The second calculation shows that
if the $\mu_j$ lie outside another (possibly larger) region, then the $L$-function
must have a zero with imaginary part less than~$14$.

For $d=3$ or $4$,
those two calculations do not resolve whether or not there is an $L$-function
with a first zero higher than the Riemann zeta function, because there
remains a very small region which could possibly correspond to an $L$-function.
And if such an $L$-function exists, one would then need to calculate its first
zero to check if it was higher than 14.13.
For the case of $d=3$, calculations of Bian~\cite{Bi, FKL} show that
there are no $L$-functions in the missing region.  Details are given by
Bober~\cite{B}.

But for $d=4$, calculations in~\cite{FKL} suggest that there is an
$L$-function in the missing region.  This is somewhat surprising because
that region is very small, as shown in Figure~\ref{fig:excluded}.
Furthermore, as was shown in Figure~\ref{fig:degree4}, that $L$-function
has a larger gap between its zeros than does the Riemann zeta function.
Note:  the example from \cite{FKL} was found prior to our implementation
of Miller's inequalities for $d=4$.  Specifically, that $L$-function was found from a general
search for degree-4 $L$-functions, not merely from an attempt to find
examples of $L$-functions with a high lowest zero.  Perhaps that makes
it even more surprising that such an example exists.

\begin{figure}[htp]
%\begin{center}
\scalebox{0.7}[0.7]{\includegraphics{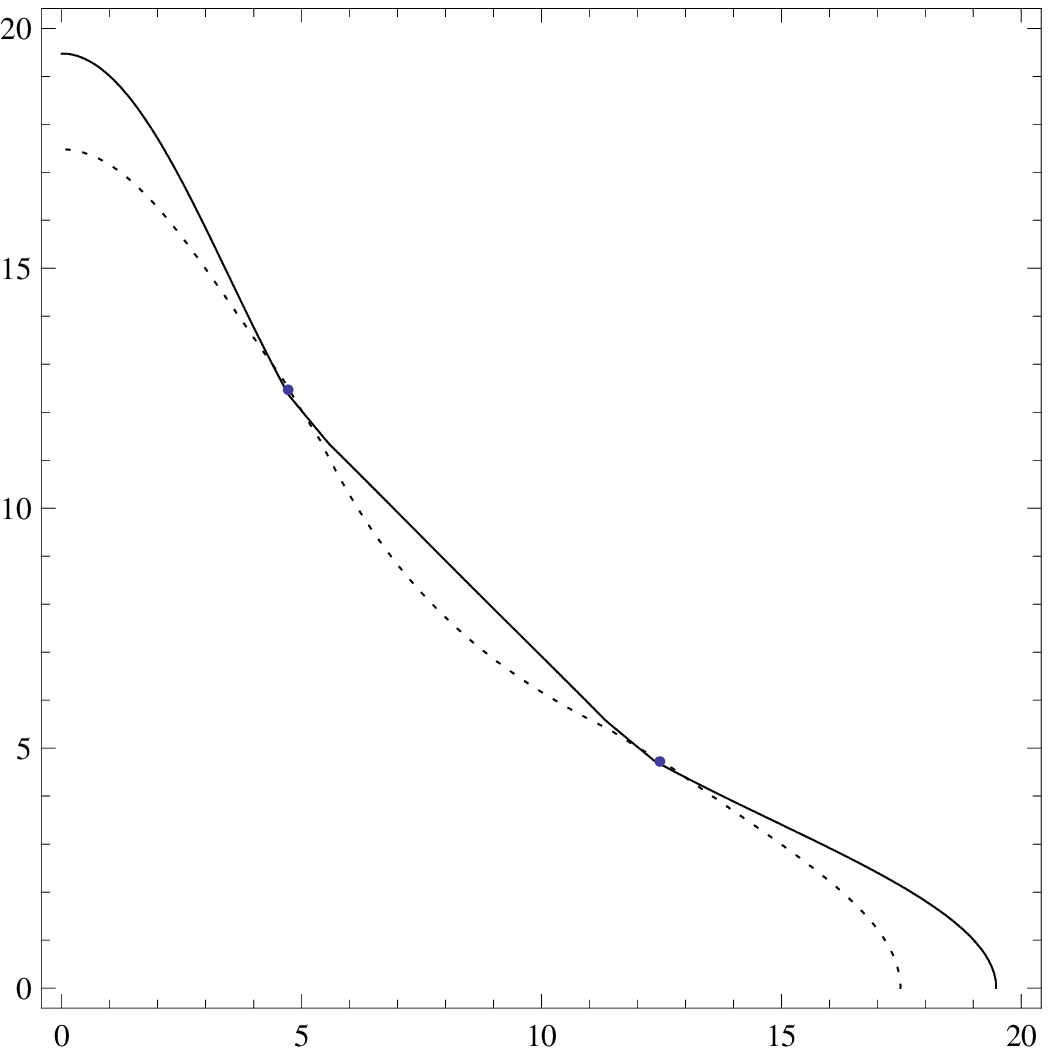}}
\hskip 0.1in
\scalebox{0.7}[0.7]{\includegraphics{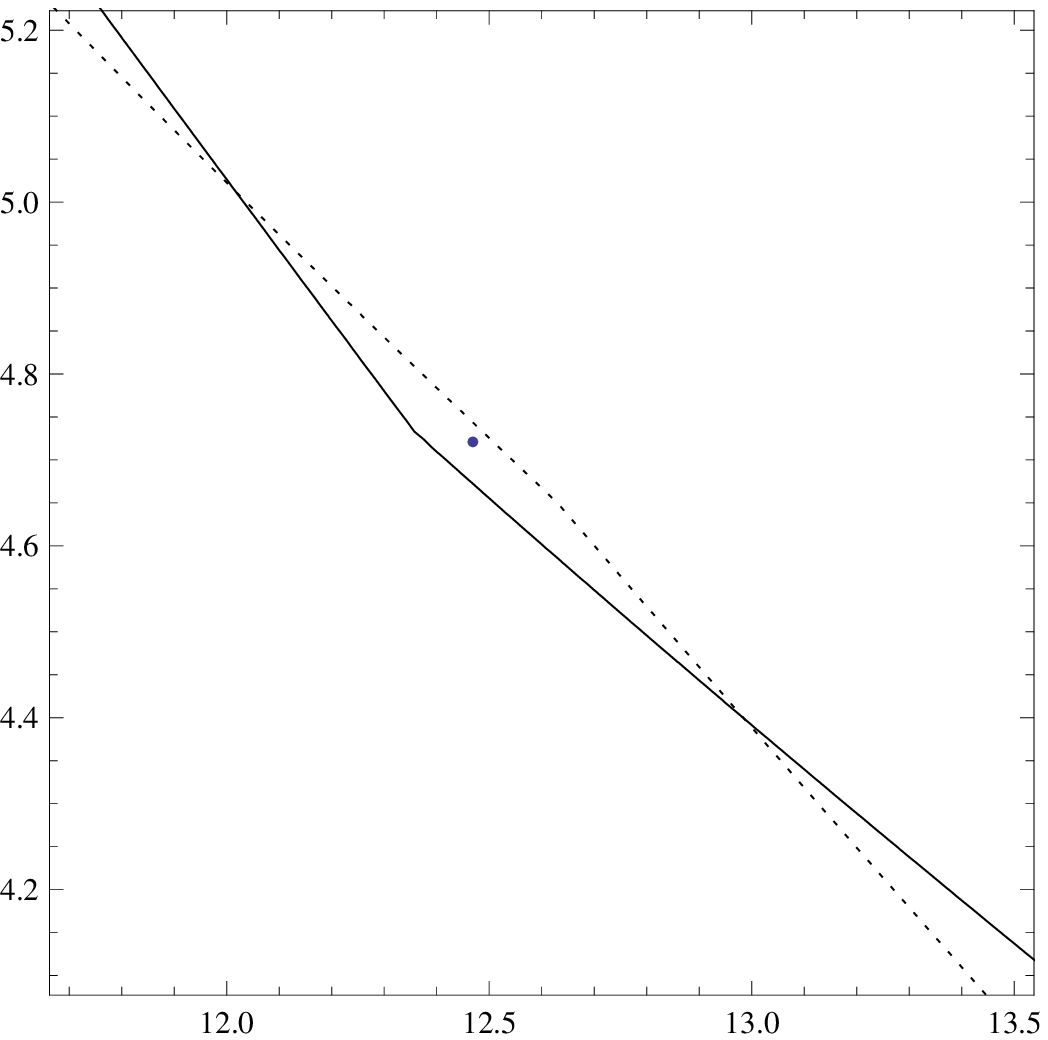}}
\caption{\sf
The region outside the solid curve describes pairs $(\nu_1,\nu_2)$ for which it
is possible that an $L$-function with functional equation \eqref{eqn:fe} exists,
where $(\mu_1,\mu_2,\mu_3,\mu_4)=(\nu_1,-\nu_1,\nu_2,-\nu_2)$.
The region outside the dotted curve describes pairs $(\nu_1,\nu_2)$ for which
such an $L$-function, if it exists, must have a zero lower than the first zero of
the Riemann zeta function.  The black dot corresponds to the $L$-function shown
in Figure~\ref{fig:degree4}. 
} \label{fig:excluded}
%\end{center}
\end{figure}

\end{document}